\documentclass[a4paper]{article}

\usepackage{amsmath,amssymb,amsthm,mathtools}
\usepackage{mathdots}
\usepackage{microtype}
\usepackage[authoryear,longnamesfirst]{natbib}
\usepackage{xcolor}
\usepackage{hyperref}
\usepackage{authblk}

\newcommand{\C}{\mathbb{C}}
\newcommand{\norm}[1]{\left\lVert #1\right\rVert_2}
\newcommand{\spec}{\sigma}
\newcommand{\Kreiss}{\mathcal{K}}
\newcommand{\Power}{\mathcal{M}}
\newcommand{\D}{\mathbb{D}}
\newcommand{\R}{\mathbb{R}}
\newcommand{\Dbar}{\overline{\mathbb{D}}}
\newcommand{\tr}{\operatorname{tr}}
\newcommand{\rank}{\operatorname{rank}}

\newtheorem{theorem}{Theorem}[section]
\newtheorem{lemma}[theorem]{Lemma}

\newtheorem{corollary}[theorem]{Corollary}
\theoremstyle{definition}
\newtheorem{definition}[theorem]{Definition}
\theoremstyle{remark}
\newtheorem{remark}[theorem]{Remark}

\AtBeginDocument{\hypersetup{
  pdftitle={Inertia-Sensitive Kreiss Bounds for $J$-Selfadjoint Matrices},
  pdfauthor={Thanh Nguyen-Cung, Binh Nguyen-Tuan}
}}

\title{Inertia-Sensitive Kreiss Bounds for \texorpdfstring{$J$}{J}--Selfadjoint Matrices}

\author[1]{Thanh Nguyen-Cung%
\thanks{24thanh.nc@vinuni.edu.vn}}

\author[1]{Binh T. Nguyen%
\thanks{binh.nt2@vinuni.edu.vn
}}

\affil[1]{VinUniversity\\
Gia Lam\\
Hanoi\\
Vietnam}

\date{}

\begin{document}

\maketitle

\begin{abstract}
Let \(A\in\C^{n\times n}\) have spectrum in the closed unit disk. Its maximal power growth \(\Power(A):=\sup_{k\ge0}\norm{A^k}\) measures transient amplification, whereas the Kreiss constant \(\Kreiss(A):=\sup_{|z|>1}(|z|-1)\norm{(zI-A)^{-1}}\) measures the corresponding resolvent growth outside the disk. The classical finite-dimensional Kreiss theorem gives \(\Power(A)\le en\Kreiss(A)\), and the linear dependence on \(n\) is unavoidable for general matrices. We show that, for matrices selfadjoint with respect to an indefinite metric, the ambient dimension $n$ can be replaced by an effective dimension determined by the minimal polynomial and the inertia of the metric. Specifically, if \(A^*J=JA\), where \(J\) is a fundamental symmetry with inertia \((n-q,q)\), then $\Power(A) \le e\min\{d(A),2q+1,2(n-q)+1\}\Kreiss(A),$ where \(d(A)\) is the degree of the minimal polynomial. Our proof requires no assumption on diagonalizability or reality on the spectrum. Instead, we associate each cyclic orbit with a finite-rank selfadjoint Hankel operator and transfer its rank and inertia to a coefficient estimate. Examples based on scaled nilpotent shifts show that the linear dependence on the smaller inertia index is asymptotically sharp, even when this index is negligible relative to the matrix size and \(d(A)=n\). We also obtain scaled-disk decay estimates and weighted-norm extensions to arbitrary nonsingular Hermitian metrics.
\end{abstract}



\section{Introduction}\label{sec:introduction}

Let $\D:=\{z\in\C:|z|<1\}$ and
$\Dbar:=\{z\in\C:|z|\le1\}$. Throughout, let
$A\in\C^{n\times n}$ satisfy $\spec(A)\subset\Dbar$ where
$\sigma(A)$ denotes the spectrum of $A$. The sequence of powers $A^k$ describes the evolution of the discrete-time system $x_{k+1}=Ax_k$. Its maximal power growth is
\(\Power(A):=\sup_{k\ge0}\norm{A^k}\). This quantity measures the largest amplification that can occur over all initial states and all time indices. The spectral condition $\spec(A)\subset\Dbar$ is the basic spectral stability condition, but it does not control finite-time behavior. For example, a nonnormal matrix may have large powers even when all eigenvalues lie strictly inside the unit disk, while a nontrivial Jordan block associated with an eigenvalue on the unit circle makes $\Power(A)$ infinite. Thus, a useful stability estimate must contain information that is not visible from the spectrum alone.

The discrete-time Kreiss constant provides such information. It is defined by \(\Kreiss(A):=\sup_{|z|>1}(|z|-1)\norm{(zI-A)^{-1}}\). This constant quantifies resolvent growth outside the unit disk and gives a frequency-domain description of transient amplification. The resolvent expansion gives the elementary lower bound $\Kreiss(A)\le\Power(A)$. The main problem is therefore to obtain a converse estimate that recovers power growth from resolvent growth. The finite-dimensional Kreiss matrix theorem gives
\begin{equation}\label{eq:classical-kreiss-intro}
\Power(A)\le en\Kreiss(A).
\end{equation}

\citet{kreiss1962} introduced the resolvent condition in the stability analysis of finite-difference approximations. \citet{tadmor1981} showed that this condition yields a power bound depending at most linearly on the matrix dimension. \citet{levequeTrefethen1984} subsequently obtained the bound $\Power(A)\le 2en\Kreiss(A)$ and conjectured that the factor 2 could be removed. \citet{spijker1991} proved this conjecture, yielding the classical estimate \eqref{eq:classical-kreiss-intro}. The dimension dependence is essentially unavoidable for general matrices: the ratio \(\Power(A)/\Kreiss(A)\) can grow almost linearly with $n$ \citep{spijkerTracognaWelfert2003}. Any improvement of the classical dimension factor must therefore exploit additional structure.

Several refinements of the Kreiss matrix theorem exploit additional algebraic, geometric, or analytic structure. Let $d(A)$ denote the degree of the minimal polynomial of $A$. A basic algebraic reduction is obtained from cyclic subspaces. For \(u\in\C^n\), let
\(E_u:=\operatorname{span}\{u,Au,A^2u,\ldots\}\), equipped with the
Euclidean norm inherited from \(\C^n\). Then, \(E_u\) is invariant under
\(A\), \(\dim E_u\le d(A)\), and \((zI-A)^{-1}\) leaves \(E_u\) invariant
for \(|z|>1\). Hence
\(\Kreiss(A|_{E_u})\le\Kreiss(A)\). Applying the classical Kreiss theorem
on \(E_u\) gives, for every unit vector \(u\),
\[
\sup_{k\ge0}\norm{A^ku}
\le
e\dim(E_u)\Kreiss(A|_{E_u})
\le
ed(A)\Kreiss(A).
\]
Taking the supremum over \(u\) yields
\(\Power(A)\le ed(A)\Kreiss(A)\). 
Thus, the minimal-polynomial degree already replaces the ambient dimension through an elementary cyclic-subspace reduction.
For families of matrices of unbounded order, \citet{gorelickKranzer1976} obtained uniform stability results under a uniform bound on the degrees of the minimal polynomials. Related operator-theoretic work established functional-calculus bounds for algebraic Kreiss operators that also depend linearly on the minimal-polynomial degree \citep{vitse2005,nikolski2006}. These results identify this degree as a natural algebraic complexity parameter. In the present paper, it enters through the dimensions of cyclic subspaces.

Further extensions of the Kreiss matrix theorem concern Cesàro bounds, matrix families of nonuniformly bounded order, and more general spectral domains \citep{strikwerdaWade1991,spijkerStraetemans1996,strikwerdaWade1997,lenferinkSpijker1991,tohTrefethen1999,raouafi2018}; see also the stability-oriented accounts \citep{vanDorsselaerEtAl1993,borovykhSpijker2000}. Of particular relevance to dimension dependence, \citet{nikolski2014} obtained sublinear growth in the matrix dimension for operators with unimodular spectrum acting on uniformly convex Banach spaces. These results rely on assumptions different from the indefinite selfadjoint structure considered here.

The present paper considers a classical class from indefinite linear algebra. Let $J=J^*=J^{-1}$ be a fundamental symmetry with inertia $(n-q,q)$, and assume that $A^*J=JA$. Then, $A$ is selfadjoint with respect to the indefinite inner product induced by $J$, or equivalently $A=JH$ for a Hermitian matrix $H$. This condition does not imply normality in the Euclidean inner product. For example, if
\[
J=\begin{pmatrix}1&0\\0&-1\end{pmatrix},
\qquad
A=\begin{pmatrix}1&-1\\1&0\end{pmatrix},
\]
then $A^*J=JA$, whereas $A^*A\ne AA^*$. Thus, Euclidean normality cannot be used to control the powers of \(A\). The sharpness examples in Section~\ref{sec:sharpness} show that large power and resolvent growth can indeed occur within this class. Classical canonical-form theory shows that the positive and negative indices of $J$ constrain the generalized eigenspaces associated with nonreal eigenvalues and the possible Jordan structures at real eigenvalues \citep{azizovIokhvidov1989,gohbergLancasterRodman2005,lancasterRodman2005}.

Finite-rank selfadjoint Hankel operators and moment sequences with finitely many negative squares have been studied in indefinite moment problems and Pontryagin-space realization theory \citep{bergChristensenMaserick1988,yafaev2015,derkachHassiDeSnoo2012}. The connection here is that each cyclic orbit of $A$ defines a finite-rank selfadjoint Hankel operator $H_\rho=W_\rho^*JW_\rho \ $, whose rank is controlled by the cyclic dimension, and whose positive and negative indices are controlled by the inertia of $J$. This transfer of structural information is the main mechanism behind the effective-dimension bound.

Existing finite-dimensional Kreiss bounds do not quantify how the positive and negative indices of an indefinite
metric affect power growth for matrices that are selfadjoint with respect
to that metric. The main question of this paper is therefore \textit{whether the a quantity can replace ambient dimension in the Kreiss estimate depending on the inertia of \(J\).}

We answer this question affirmatively. Define
$\nu_J(A):=\min\{d(A),2q+1,2(n-q)+1\}$.
Theorem~\ref{thm:main} proves that, for every \(k\ge1\),
\begin{equation}\label{eq:main-pointwise-intro}
\norm{A^k}
\le
\nu_J(A)\left(1+\frac1k\right)^k\Kreiss(A).
\end{equation}
We shall also use the classical time-index estimate (Lemma~\ref{lem:time-index}), which is the unit-disk specialization of the Faber-polynomial Kreiss estimate of \citet{tohTrefethen1999}; see also \citet{raouafi2018}.Together with Theorem~\ref{thm:main}, this yields Corollary~\ref{cor:time-inertia}:
\begin{equation}\label{eq:main-intro}
\norm{A^k}
\le
e\min\{k+1,\nu_J(A)\}\Kreiss(A),
\qquad k\ge0.
\end{equation}
The term \(d(A)\) recovers the cyclic-subspace reduction discussed above. The new feature is the additional control by \(2q+1\) and
\(2(n-q)+1\). Since \(d(A)\le n\), one has \(\nu_J(A)\le n\), the result is not weaker than the classical finite-dimensional estimate.
It improves the minimal-polynomial bound whenever
$2\min\{q,n-q\}+1<d(A)$. 
In particular, the effective-dimension factor remains bounded along a
sequence of growing matrices if the smaller inertia index remains
bounded. If \(|(n-q)-q|\le1\) and \(d(A)=n\), then \(\nu_J(A)=n\), and the theorem recovers the order of the classical bound. 

Importantly, this refinement requires neither diagonalizability nor a real spectrum. The proof uses the selfadjoint Hankel pullback $H_\rho=W_\rho^*JW_\rho$: its rank is controlled by the cyclic dimension, while its positive and negative indices are controlled by those of $J$. Lemma~\ref{lem:hankel-effective} converts these three structural bounds into a coefficient estimate for $u^*JA^ku$, and the resolvent controls the corresponding boundary generating function. A natural remaining question is \textit{whether the linear dependence on the smaller inertia index can be improved.}

The sharpness results show that it cannot, in general.  The sharpness construction is based on the classical scaled nilpotent shift \citep{levequeTrefethen1984,spijkerTracognaWelfert2003}. Theorem~\ref{thm:sharpness} realizes this core as a $J$-selfadjoint matrix with inertia $(q+1,q)$ and shows that both $\Power(A)$ and $\Kreiss(A)$ are preserved under zero padding. The core construction has both \(d(A)=2q+1\) and smaller-inertia factor \(2q+1\), therefore, it does not by itself distinguish the inertia improvement from the minimal-polynomial reduction. Corollary~\ref{cor:sharpness-full-degree} resolves this issue by replacing the zero block with a diagonal Hermitian contraction having distinct nonzero eigenvalues. The resulting examples satisfy $d(A)=n$; moreover, ambient dimension can be chosen so that \(q=o(n)\), while $\nu_J(A) = 2q + 1$.  Thus, the linear dependence on the smaller inertia index remains asymptotically sharp even when the minimal-polynomial degree equals the ambient dimension.

Two further consequences extend the main estimate. For $\gamma \in (0, 1)$, if $\rho(A)<\gamma$, Corollary~\ref{cor:geometric-decay} gives a geometric decay bound in terms of the scaled Kreiss constant $\Kreiss_\gamma(A)$. Corollary~\ref{cor:metric} transfers the result from a fundamental symmetry to an arbitrary nonsingular Hermitian metric through a canonical similarity transformation and the corresponding weighted norm. Both extensions preserve the same effective-dimension formula. 

The paper is organized as follows. Section~\ref{sec:preliminaries} introduces the notation and basic definitions. Section~\ref{sec:hankel} proves the Hankel coefficient estimate. Section~\ref{sec:kreiss} establishes the pointwise and uniform inertia-sensitive Kreiss bounds. Section~\ref{sec:sharpness} constructs asymptotically sharp examples and shows that the linear dependence on the smaller inertia index cannot be improved. Section~\ref{sec:consequences} gives the scaled-disk and general-metric consequences.

\section{Notation and basic definitions}\label{sec:preliminaries}

For vectors in $\C^n$ and sequences in $\ell^2(\mathbb N_0)$, the notation $\norm{\cdot}$ denotes the Euclidean norm and the $\ell^2$-norm, respectively. For matrices, $\norm{\cdot}$ denotes the induced spectral norm. The spectrum and spectral radius of a matrix $A$ are denoted by $\spec(A)$ and $\rho(A)$, respectively. We write
\[
\D:=\{z\in\C:|z|<1\},\qquad \Dbar:=\{z\in\C:|z|\le1\}.
\]
The Hilbert-space inner product is linear in its first argument. For a bounded operator $T$ on $\ell^2(\mathbb N_0)$, the notation $\lVert T\rVert$ denotes the operator norm induced by the $\ell^2$-norm. If $T$ is trace class, then $\lVert T\rVert_1$ denotes its trace norm.

Let $\mathbb N_0:=\{0,1,2,\ldots\}$, and let $(e_j)_{j\ge0}$ denote the standard orthonormal basis of $\ell^2(\mathbb N_0)$.

\begin{definition}\label{def:hankel-operator}
A bounded operator $H$ on $\ell^2(\mathbb N_0)$ is called a \emph{Hankel operator} if there exists a sequence $(b_m)_{m\ge0}$ such that
\[
\langle He_j,e_i\rangle=b_{i+j},
\qquad i,j\ge0.
\]
Equivalently, the matrix of $H$ in the standard basis has the form $H=[b_{i+j}]_{i,j\ge0}$,
so its entries are constant along the anti-diagonals. The associated power series
\[
F_H(w):=\sum_{m\ge0}b_mw^m,
\qquad w\in\D,
\]
is called the \emph{analytic generating function} of $H$. Since $(b_m)_{m\ge0}$ is the first column of the matrix of $H$, it belongs to $\ell^2(\mathbb N_0)$, and hence the series defining $F_H$ converges absolutely at every point of $\D$.

The Hankel operator $H$ is selfadjoint if and only if $b_m\in\R$ for every $m\ge0$. If $H$ is compact and selfadjoint, we denote by $\kappa_-(H)$ and $\kappa_+(H)$ the numbers of its negative and positive eigenvalues, respectively, counted with multiplicity.
\end{definition}

\begin{definition}\label{def:J-selfadjoint}
A matrix $J\in\C^{n\times n}$ is a \emph{fundamental symmetry} if $J=J^*=J^{-1}$. Its eigenvalues belong to $\{-1,1\}$. We denote the numbers of negative and positive eigenvalues by $\operatorname{ind}_-(J)$ and $\operatorname{ind}_+(J)$, respectively. We write $q:=\operatorname{ind}_-(J)=\dim\ker(J+I)$, so the inertia of $J$ is $(n-q,q)$. A subspace $L\subset\C^n$ is called $J$-negative if $x^*Jx<0$ for every nonzero $x\in L$, and it is called $J$-positive if $x^*Jx>0$ for every nonzero $x\in L$. By the spectral theorem, $\operatorname{ind}_-(J)$ and $\operatorname{ind}_+(J)$ are the maximal dimensions of $J$-negative and $J$-positive subspaces, respectively. A matrix $A\in\C^{n\times n}$ is \emph{$J$-selfadjoint} if $A^*J=JA$. Equivalently, $A$ is selfadjoint with respect to the indefinite inner product $[x,y]_J:=y^*Jx$.
\end{definition}

The identity $A^*J=JA$ is also equivalent to the existence of a Hermitian matrix $H$ such that $A=JH$, namely $H=JA$. The same matrix $A$ is also $(-J)$--selfadjoint. Since the negative index of $-J$ is $n-q$, a one-sided inertia estimate can be applied with either signature index. For standard background on finite-dimensional indefinite inner-product spaces, see \citet{bognar1974}.

\begin{definition}\label{def:kreiss}
Let $A\in\C^{n\times n}$ satisfy $\spec(A)\subset\Dbar$. Its \emph{maximal power growth} and \emph{discrete-time Kreiss constant} are
\[
\Power(A):=\sup_{k\ge0}\norm{A^k},\qquad \Kreiss(A):=\sup_{|z|>1}(|z|-1)\norm{(zI-A)^{-1}}.
\]
We allow either quantity to take the value $+\infty$.
\end{definition}

The resolvent expansion gives $\Kreiss(A)\le\Power(A)$. Moreover, $\Kreiss(A)\ge1$, because
\[
(r-1)\norm{(rI-A)^{-1}}\longrightarrow1\qquad (r\to\infty).
\]
The inclusion $\spec(A)\subset\Dbar$ does not imply $\Kreiss(A)<\infty$. For example, a nontrivial Jordan block at a unimodular eigenvalue makes the constant infinite. All estimates below are interpreted in the extended sense in this case.

\begin{definition}\label{def:effective-dimension}
Let $m_A$ denote the minimal polynomial of $A$, and set
\[
d(A):=\deg m_A,\qquad \nu_J(A):=\min\{d(A),2q+1,2(n-q)+1\}.
\]
For $\rho(A)<\gamma$, define
\[
\Kreiss_\gamma(A):=\sup_{|z|>\gamma}(|z|-\gamma)\norm{(zI-A)^{-1}}.
\]
\end{definition}

The strict inequality $\rho(A)<\gamma$ implies $\Kreiss_\gamma(A)<\infty$. The quantity $d(A)$ bounds the dimension of every cyclic subspace generated by $A$. This is the form in which the minimal polynomial enters the proof.

\section{Coefficient estimate for selfadjoint Hankel operators}
\label{sec:hankel}

The following lemma is the main analytic ingredient in the proof of
Theorem~\ref{thm:main}. Background on finite-rank and trace-class Hankel
operators can be found in
\citet{bonsallWalsh1986,peller2003,yafaev2015}; see also
\citet{simon2005} for trace ideals. Connections with Hankel moment
sequences, finite negative squares, and indefinite moment problems are
developed in
\citet{bergChristensenMaserick1988,derkachHassiDeSnoo2012}.

\begin{lemma}
\label{lem:hankel-effective}
Let $H=[b_{i+j}]_{i,j\ge0}$ be a finite-rank selfadjoint Hankel operator on
$\ell^2(\mathbb N_0)$, and let $F(w):=\sum_{m\ge0}b_mw^m$. Suppose that
$F$ extends continuously to $\Dbar$. Then, for every $k\ge0$,
\[
(k+1)|b_k|
\le
\min\{\rank(H),2\kappa_-(H)+1,2\kappa_+(H)+1\}
\max_{|w|=1}|F(w)|.
\]
\end{lemma}

\begin{proof}
Set $M:=\max_{|w|=1}|F(w)|$. We first show that $M$ controls the operator
norm of $H$.

For every $m\ge0$ and $0<r<1$, Cauchy's coefficient formula gives
\[
b_mr^m=\frac{1}{2\pi}\int_0^{2\pi}
F(re^{\mathrm{i}t})e^{-\mathrm{i}mt}\,dt.
\]
Since $F$ is continuous on the compact set $\Dbar$, the functions
$F(re^{\mathrm{i}t})$ converge uniformly to $F(e^{\mathrm{i}t})$ as
$r\to1^-$. Therefore,
\begin{equation}\label{eq:hankel-fourier-coefficient}
b_m=\frac{1}{2\pi}\int_0^{2\pi}
F(e^{\mathrm{i}t})e^{-\mathrm{i}mt}\,dt.
\end{equation}

Let $x=(x_j)_{j\ge0}$ and $y=(y_i)_{i\ge0}$ be finitely supported, and
define
\[
X(e^{\mathrm{i}t}):=\sum_{j\ge0}x_je^{-\mathrm{i}jt},
\qquad
Y(e^{\mathrm{i}t}):=\sum_{i\ge0}y_ie^{\mathrm{i}it}.
\]
Using \eqref{eq:hankel-fourier-coefficient} with $m=i+j$, and recalling
that the inner product is linear in its first argument, we obtain
\[
\begin{aligned}
\langle Hx,y\rangle
&=\sum_{i,j\ge0}b_{i+j}x_j\overline{y_i} \\
&=\frac{1}{2\pi}\int_0^{2\pi}
F(e^{\mathrm{i}t})
\left(\sum_{j\ge0}x_je^{-\mathrm{i}jt}\right)
\left(\sum_{i\ge0}\overline{y_i}e^{-\mathrm{i}it}\right)\,dt \\
&=\frac{1}{2\pi}\int_0^{2\pi}
F(e^{\mathrm{i}t})X(e^{\mathrm{i}t})
\overline{Y(e^{\mathrm{i}t})}\,dt.
\end{aligned}
\]
The interchange of the sums and the integral is valid because $x$ and
$y$ are finitely supported. Since $|F(e^{\mathrm{i}t})|\le M$, the
Cauchy--Schwarz inequality gives
\[
|\langle Hx,y\rangle|
\le
M
\left(\frac{1}{2\pi}\int_0^{2\pi}|X(e^{\mathrm{i}t})|^2\,dt\right)^{1/2}
\left(\frac{1}{2\pi}\int_0^{2\pi}|Y(e^{\mathrm{i}t})|^2\,dt\right)^{1/2}.
\]
By Parseval's identity,
\[
\frac{1}{2\pi}\int_0^{2\pi}|X(e^{\mathrm{i}t})|^2\,dt=\norm{x}^2,
\qquad
\frac{1}{2\pi}\int_0^{2\pi}|Y(e^{\mathrm{i}t})|^2\,dt=\norm{y}^2.
\]
Therefore, $|\langle Hx,y\rangle|\le M\norm{x}\norm{y}$. Since finitely
supported sequences are dense in $\ell^2(\mathbb N_0)$, the estimate
extends to all $x,y\in\ell^2(\mathbb N_0)$. Taking the supremum over unit
vectors $x$ and $y$ gives
\begin{equation}\label{eq:hankel-operator-norm}
\lVert H\rVert
=
\sup_{\norm{x}=\norm{y}=1}|\langle Hx,y\rangle|
\le M.
\end{equation}

We next estimate the trace of $H$. Since $H$ has finite rank, it is trace
class. In the standard orthonormal basis,
\[
\sum_{j\ge0}|b_{2j}|
=
\sum_{j\ge0}|\langle He_j,e_j\rangle|
\le
\lVert H\rVert_1.
\]
Thus $\sum_{j\ge0}b_{2j}$ converges absolutely, and
$\tr H=\sum_{j\ge0}b_{2j}$. Moreover, for every $0<r<1$,
\[
F(r)+F(-r)=2\sum_{j\ge0}b_{2j}r^{2j}.
\]
Since $\sum_{j\ge0}|b_{2j}|<\infty$, dominated convergence gives
\begin{equation}\label{eq:hankel-trace}
\tr H=\frac{F(1)+F(-1)}{2}.
\end{equation}
It follows that $|\tr H|\le M$.

We now derive three bounds for the trace norm of $H$. The rank bound
follows from
\[
\lVert H\rVert_1\le\rank(H)\lVert H\rVert\le\rank(H)M.
\]

For the estimate involving the negative index, let $\lambda_j(H)$ denote
the nonzero eigenvalues of $H$, counted with multiplicity. Since $H$ is
selfadjoint,
\[
\lVert H\rVert_1
=
\tr H+2\sum_{\lambda_j(H)<0}|\lambda_j(H)|.
\]
There are $\kappa_-(H)$ negative eigenvalues, counted with multiplicity,
and every eigenvalue satisfies $|\lambda_j(H)|\le\lVert H\rVert$. Hence
\[
\lVert H\rVert_1
\le
|\tr H|+2\kappa_-(H)\lVert H\rVert
\le
\bigl(2\kappa_-(H)+1\bigr)M.
\]
Applying the same argument to $-H$, whose negative index is
$\kappa_+(H)$ and whose generating function is $-F$, gives
\[
\lVert H\rVert_1\le\bigl(2\kappa_+(H)+1\bigr)M.
\]
Combining the three estimates, we obtain
\begin{equation}\label{eq:hankel-trace-norm-effective}
\lVert H\rVert_1
\le
\min\{\rank(H),2\kappa_-(H)+1,2\kappa_+(H)+1\}M.
\end{equation}

It remains to recover $b_k$ from a trace identity. For $k\ge0$, define
\[
R_ke_j=
\begin{cases}
e_{k-j}, & 0\le j\le k,\\
0, & j>k.
\end{cases}
\]
The operator $R_k$ is selfadjoint, has finite rank, and satisfies
$\lVert R_k\rVert=1$. Since $R_kH$ is trace class,
\[
\begin{aligned}
\tr(R_kH)
&=\sum_{j\ge0}\langle R_kHe_j,e_j\rangle
=\sum_{j\ge0}\langle He_j,R_ke_j\rangle \\
&=\sum_{j=0}^k\langle He_j,e_{k-j}\rangle
=(k+1)b_k.
\end{aligned}
\]
Therefore,
\[
(k+1)|b_k|
=
|\tr(R_kH)|
\le
\lVert R_kH\rVert_1
\le
\lVert R_k\rVert\lVert H\rVert_1
=
\lVert H\rVert_1.
\]
Using \eqref{eq:hankel-trace-norm-effective} and the definition of $M$
proves the result.
\end{proof}

\begin{remark}
\label{rem:hankel-novelty}
Finite-rank selfadjoint Hankel operators and Hankel moment sequences with
finitely many negative squares have a well-developed structural theory;
see, for example,
\citet{bergChristensenMaserick1988,yafaev2015,
derkachHassiDeSnoo2012}. Lemma~\ref{lem:hankel-effective} combines the
symbol bound $\lVert H\rVert\le\max_{|w|=1}|F(w)|$, the trace identity
\eqref{eq:hankel-trace}, and rank- and inertia-sensitive estimates for the
trace norm.

Its role in this paper is that the orbit pullback
$H_\rho=W_\rho^*JW_\rho$ inherits its rank from the cyclic subspace
generated by the orbit, while its positive and negative indices are
controlled by the two inertia indices of $J$.
\end{remark}

\section{Inertia-sensitive Kreiss bounds}\label{sec:kreiss}

\begin{lemma}\label{lem:time-index}
Let $A\in\C^{n\times n}$ satisfy $\spec(A)\subset\Dbar$. Then, for every $k\ge0$,
\[
\norm{A^k}\le (k+1)\left(1+\frac1{k}\right)^{k}\Kreiss(A)<e(k+1)\Kreiss(A).
\]
\end{lemma}

\begin{proof}
There is nothing to prove when $\Kreiss(A)=\infty$, so assume that $\Kreiss(A)<\infty$. For any $r>1$, Cauchy's formula gives
\[
A^k=\frac1{2\pi\mathrm{i}}\int_{|z|=r}z^k(zI-A)^{-1}\,dz.
\]
Taking norms and using the definition of $\Kreiss(A)$, we obtain
\[
\norm{A^k}\le r^{k+1}\max_{|z|=r}\norm{(zI-A)^{-1}}\le \frac{r^{k+1}}{r-1}\Kreiss(A).
\]
Choose $r=1+1/k$. Then
\[
\frac{r^{k+1}}{r-1}=(k+1)\left(1+\frac1{k}\right)^{k}<e(k+1),
\]
which proves the lemma.
\end{proof}

We now prove the main estimate directly. The proof associates a Hankel operator with each cyclic orbit and uses its rank and inertia to control the corresponding scalar moments.

\begin{theorem}\label{thm:main}
Let $J=J^*=J^{-1}$ have inertia $(n-q,q)$. Suppose that $A^*J=JA$ and $\spec(A)\subset\Dbar$. Then, for every $k\ge1$,
\begin{equation}\label{eq:pointwise-main}
\norm{A^k}\le\nu_J(A)\left(1+\frac1k\right)^k\Kreiss(A).
\end{equation}
Consequently,
\begin{equation}\label{eq:uniform-main}
\Power(A)\le e\nu_J(A)\Kreiss(A).
\end{equation}
\end{theorem}

\begin{proof}
The conclusion is immediate if $\Kreiss(A)=\infty$, so assume that $\Kreiss(A)<\infty$. Fix $k\ge1$, a unit vector $u\in\C^n$, and a number $\rho>1$. For $m\ge0$, define $a_m:=u^*JA^mu$.

We first note that every $a_m$ is real. Indeed, the identity $A^*J=JA$ implies by induction that $(A^*)^mJ=JA^m$. Hence $(JA^m)^*=(A^*)^mJ=JA^m$, so $JA^m$ is Hermitian and therefore $a_m=u^*JA^mu\in\R$.

Next, choose $s$ such that $1<s<\rho$. By the spectral radius formula, $\limsup_{j\to\infty}\norm{A^j}^{1/j}=\rho(A)\le1<s$. Thus there exists a constant $C_s>0$ such that $\norm{A^j}\le C_ss^j$ for all $j\ge0$. Define an operator $W_\rho$ on the standard basis of $\ell^2(\mathbb N_0)$ by $W_\rho e_j:=\rho^{-j}A^ju$. Then
\[
\sum_{j\ge0}\norm{W_\rho e_j}^2=\sum_{j\ge0}\rho^{-2j}\norm{A^ju}^2\le C_s^2\sum_{j\ge0}\left(\frac{s}{\rho}\right)^{2j}<\infty.
\]
For a finitely supported sequence $x=(x_j)_{j\ge0}$, the Cauchy--Schwarz inequality gives
\[
\norm{W_\rho x}\le\sum_{j\ge0}|x_j|\norm{W_\rho e_j}\le\norm{x}\left(\sum_{j\ge0}\norm{W_\rho e_j}^2\right)^{1/2}.
\]
Therefore $W_\rho$ extends uniquely to a bounded operator from $\ell^2(\mathbb N_0)$ to $\C^n$.

Set $H_\rho:=W_\rho^*JW_\rho$. For $i,j\ge0$, we have
\[
\begin{aligned}
\langle H_\rho e_j,e_i\rangle
&=\langle JW_\rho e_j,W_\rho e_i\rangle
 =\frac{(A^iu)^*JA^ju}{\rho^{i+j}}\\
&=\frac{u^*(A^*)^iJA^ju}{\rho^{i+j}}
 =\frac{u^*JA^{i+j}u}{\rho^{i+j}}
 =\frac{a_{i+j}}{\rho^{i+j}}.
\end{aligned}
\]
Thus $H_\rho$ is the finite-rank selfadjoint Hankel operator
\begin{equation}\label{eq:hankel-pullback}
H_\rho=\left[\frac{a_{i+j}}{\rho^{i+j}}\right]_{i,j\ge0}.
\end{equation}

We now record the three structural bounds needed for Lemma~\ref{lem:hankel-effective}. First, the range of $W_\rho$ is contained in the cyclic subspace
\[
E_u:=\operatorname{span}\{u,Au,A^2u,\ldots\}.
\]
Since $\dim E_u\le d(A)$, it follows that
\[
\rank H_\rho\le\rank W_\rho\le d(A).
\]
Second, let $L\subset\ell^2(\mathbb N_0)$ be a subspace on which the quadratic form of $H_\rho$ is negative definite. For every nonzero $x\in L$,
\[
\langle H_\rho x,x\rangle=(W_\rho x)^*J(W_\rho x)<0.
\]
Hence $W_\rho x\ne0$, so $W_\rho$ is injective on $L$, and $W_\rho L$ is a $J$-negative subspace. Since $\operatorname{ind}_-(J)=q$ is the maximal dimension of a $J$-negative subspace, we have $\dim L\le q$. For a finite-rank selfadjoint operator, the number of negative eigenvalues, counted with multiplicity, equals the maximal dimension of a subspace on which its quadratic form is negative definite. Therefore $\kappa_-(H_\rho)\le q$. The same argument applied to positive-definite subspaces gives $\kappa_+(H_\rho)\le n-q$.

Define the analytic generating function of $H_\rho$ by
\[
F_\rho(w):=\sum_{m\ge0}\frac{a_m}{\rho^m}w^m.
\]
Because $|a_m|\le\norm{A^m}\le C_ss^m$ and $s<\rho$, this series converges absolutely and uniformly on $\Dbar$. Applying Lemma~\ref{lem:hankel-effective} to \eqref{eq:hankel-pullback} and using the three bounds above, we obtain
\begin{equation}\label{eq:hankel-moment-bound}
(k+1)\frac{|a_k|}{\rho^k}\le\nu_J(A)\max_{|w|=1}|F_\rho(w)|.
\end{equation}

It remains to estimate the boundary values of $F_\rho$ by the resolvent. Fix $w$ with $|w|=1$ and set $z:=\rho/w$, so $|z|=\rho>1$. Since $\norm{A^m}\le C_ss^m$ and $s<\rho=|z|$, the series $\sum_{m\ge0}A^m/z^{m+1}$ converges absolutely in operator norm. Hence the resolvent expansion is
\[
(zI-A)^{-1}=\sum_{m\ge0}\frac{A^m}{z^{m+1}}.
\]
Using $z=\rho/w$, we compute
\[
\begin{aligned}
\frac{\rho}{w}u^*J(zI-A)^{-1}u
&=\frac{\rho}{w}\sum_{m\ge0}\frac{u^*JA^mu}{z^{m+1}}
 =\frac{\rho}{w}\sum_{m\ge0}a_m\left(\frac{w}{\rho}\right)^{m+1}\\
&=\sum_{m\ge0}\frac{a_m}{\rho^m}w^m
 =F_\rho(w).
\end{aligned}
\]
Since $J$ is unitary and $\norm{u}=1$, this identity gives
\[
|F_\rho(w)|\le\rho\norm{(zI-A)^{-1}}.
\]
Taking the maximum over $|w|=1$, or equivalently over $|z|=\rho$, yields
\begin{equation}\label{eq:generating-resolvent-bound}
\max_{|w|=1}|F_\rho(w)|\le\rho\max_{|z|=\rho}\norm{(zI-A)^{-1}}\le\frac{\rho}{\rho-1}\Kreiss(A).
\end{equation}
Combining \eqref{eq:hankel-moment-bound} and \eqref{eq:generating-resolvent-bound}, we find
\[
|u^*JA^ku|=|a_k|\le\nu_J(A)\frac{\rho^{k+1}}{(k+1)(\rho-1)}\Kreiss(A).
\]
Choose $\rho=1+1/k$. Then $(k+1)(\rho-1)=(k+1)/k$, and hence
\[
\frac{\rho^{k+1}}{(k+1)(\rho-1)}=\frac{k}{k+1}\left(1+\frac1k\right)^{k+1}=\left(1+\frac1k\right)^k.
\]
Therefore
\[
|u^*JA^ku|\le\nu_J(A)\left(1+\frac1k\right)^k\Kreiss(A).
\]
Finally, $JA^k$ is Hermitian and multiplication by the unitary matrix $J$ does not change the spectral norm. Thus
\[
\norm{A^k}=\norm{JA^k}=\sup_{\norm{u}=1}|u^*JA^ku|,
\]
and \eqref{eq:pointwise-main} follows.

For the uniform estimate, $\norm{A^0}=1\le\Kreiss(A)$. For every $k\ge1$, the inequality $(1+1/k)^k\le e$ and \eqref{eq:pointwise-main} give $\norm{A^k}\le e\nu_J(A)\Kreiss(A)$. Taking the supremum over $k\ge0$ proves \eqref{eq:uniform-main}.
\end{proof}

\begin{remark}\label{rem:local-cyclic-dimension}
For a nonzero vector $u$, let
\[
d_u(A):=\dim\operatorname{span}\{u,Au,A^2u,\ldots\}.
\]
The proof above uses only the bound $\rank H_\rho\le d_u(A)$. Hence, for every unit vector $u$ and every $k\ge1$,
\[
|u^*JA^ku|\le\min\{d_u(A),2q+1,2(n-q)+1\}\left(1+\frac1k\right)^k\Kreiss(A).
\]
This is a quadratic-form estimate along the direction $u$; the global theorem follows by replacing $d_u(A)$ with $d(A)$ and taking the supremum over unit vectors.
\end{remark}

\begin{corollary}\label{cor:time-inertia}
Under the assumptions of Theorem~\ref{thm:main}, for every $k\ge1$,
\[
\norm{A^k}\le \left(1+\frac1{k}\right)^{k} \min\left\{(k+1),\,\nu_J(A)\right\}\Kreiss(A).
\]
Consequently, for every $k\ge0$,
\[
\norm{A^k}\le e\min\{k+1,\nu_J(A)\}\Kreiss(A).
\]
\end{corollary}

\begin{proof}
For $k\ge1$, take the smaller of the bounds in Lemma~\ref{lem:time-index} and Theorem~\ref{thm:main}. The simplified estimate follows from  $(1+1/k)^k\le e$. For $k=0$, we have $\norm{A^0}=1\le\Kreiss(A)$.
\end{proof}

\begin{remark}
\label{rem:comparison}
Since \(d(A)\le n\), one has
$\nu_J(A) = \min \{d(A),2q+1,2(n-q)+1\} \le n.
$ Therefore, the uniform estimate $\Power(A)\le e\nu_J(A)\Kreiss(A)$
is never weaker than the classical finite-dimensional bound $\Power(A)\le en\Kreiss(A)$.
The effective-dimension factor is strictly smaller than \(n\) whenever
$d(A)<n$ or $2\min\{q,n-q\}+1<n$. The theorem also immediately gives the two one-sided estimates
\[
\Power(A)\le e(2q+1)\Kreiss(A),
\qquad
\Power(A)\le e\bigl(2(n-q)+1\bigr)\Kreiss(A).
\]
In particular, if
$J=\operatorname{diag}(I_p,-I_q),$ and $p+q=n$ then
\[
\Power(A)
\le
e\min\{d(A),2p+1,2q+1\}\Kreiss(A).
\]
Thus, when one of $p$ and $q$ is much smaller than the other, the dimension factor is controlled by the smaller signature block rather than by the ambient dimension.

At the definite endpoints \(q=0\) and \(q=n\), one has \(J=\pm I\), and the
relation \(A^*J=JA\) reduces to \(A=A^*\). Since
\(\spec(A)\subset\Dbar\), selfadjointness implies
$\spec(A)\subset[-1,1]$. Consequently, $\Power(A)=\sup_{k\ge0}\norm{A^k}=1$.
Moreover, for every \(|z|>1\),
\[
\norm{(zI-A)^{-1}}
=
\frac{1}{\operatorname{dist}(z,\spec(A))}
\le
\frac{1}{|z|-1},
\]
because \(\spec(A)\subset\Dbar\). Hence
$\Kreiss(A)\le 1$.
Together with the general lower bound \(\Kreiss(A)\ge1\), this gives
\[
\Power(A)=\Kreiss(A)=1.
\]
Thus the factor \(e\) is not sharp at the definite endpoints.
\end{remark}

\begin{remark}
\label{rem:spectrum}
Theorem~\ref{thm:main} does not require the spectrum of \(A\) to be real,
nor does it require \(A\) to be diagonalizable. The scalar moments $u^*JA^mu$ are real because \(JA^m\) is Hermitian, not because the eigenvalues of
\(A\) are real. If \(A\) has a nontrivial Jordan block associated with an eigenvalue on the unit circle, then its powers are unbounded and $\Kreiss(A)=\infty$.
In this case, the estimates in Theorem~\ref{thm:main} are understood in the
extended sense.
\end{remark}

\section{Sharpness of the inertia factor}
\label{sec:sharpness}

We now show that the linear dependence on the smaller inertia index in
Theorem~\ref{thm:main} cannot, in general, be improved. The construction
is based on the scaled nilpotent shift used in the classical sharpness
examples for the Kreiss matrix theorem
\citep{levequeTrefethen1984,spijkerTracognaWelfert2003}. The point specific
to the present setting is that this shift is selfadjoint with respect to
a fundamental symmetry of inertia $(q+1,q)$ in dimension $2q+1$.

\begin{theorem}
\label{thm:sharpness}
For every \(q\ge1\), set \(n:=2q+1\). There exist a fundamental symmetry
\(J_q\in\C^{n\times n}\) with negative index \(q\) and a
\(J_q\)-selfadjoint matrix \(A_q\in\C^{n\times n}\) such that
\[
\spec(A_q)=\{0\}, \qquad
d(A_q)=n, \qquad \text{and} \qquad \frac{\Power(A_q)}{en\Kreiss(A_q)}
\xrightarrow[q\to\infty]{}1.
\]

For every \(N\ge n\), this pair can be extended to dimension \(N\)
without changing either \(\Power\) or \(\Kreiss\). Hence, for every
sequence \(N(q)\ge2q+1\),
\[
\frac{\Power(A_{q,N(q)})}
{e(2q+1)\Kreiss(A_{q,N(q)})}
\xrightarrow[q\to\infty]{}1.
\]
\end{theorem}

\begin{proof}
Fix \(q\ge1\), set \(n:=2q+1\), and let \(S_n\) be the nilpotent Jordan
shift
\[
S_n=
\begin{pmatrix}
0&1&0&\cdots&0\\
0&0&1&\ddots&\vdots\\
\vdots&\ddots&\ddots&\ddots&0\\
0&\cdots&0&0&1\\
0&\cdots&\cdots&0&0
\end{pmatrix}.
\]
Let \(J_q\) be the reversal matrix
\[
J_q=
\begin{pmatrix}
0&\cdots&0&1\\
0&\cdots&1&0\\
\vdots&\iddots&\iddots&\vdots\\
1&0&\cdots&0
\end{pmatrix},
\]
and define \(\alpha_n:=n^3\) and \(A_q:=\alpha_nS_n\). The particular choice \(\alpha_n=n^3\) is made only to ensure that the lower-order terms in the resolvent expansion are asymptotically negligible compared with its highest-order term.

The matrix \(J_q\) is real symmetric and satisfies \(J_q^2=I\), so
\(J_q=J_q^*=J_q^{-1}\). For \(1\le j\le q\), the vectors
\[
\frac{e_j+e_{n+1-j}}{\sqrt2}
\quad\text{and}\quad
\frac{e_j-e_{n+1-j}}{\sqrt2}
\]
are eigenvectors of \(J_q\) with eigenvalues \(1\) and \(-1\),
respectively. The remaining vector \(e_{q+1}\) is an eigenvector with
eigenvalue \(1\). Hence \(J_q\) has inertia \((q+1,q)\), and therefore
\(\operatorname{ind}_-(J_q)=q\).

We next verify \(J_q\)-selfadjointness. Since
\(J_qe_j=e_{n+1-j}\), one has
\[
J_qS_nJ_qe_j=
\begin{cases}
e_{j+1}, & 1\le j\le n-1,\\
0, & j=n.
\end{cases}
\]
Thus \(J_qS_nJ_q=S_n^*\), or equivalently \(S_n^*J_q=J_qS_n\). It follows
that
\[
A_q^*J_q
=
\alpha_nS_n^*J_q
=
\alpha_nJ_qS_n
=
J_qA_q.
\]

Since \(S_n\) is a single nilpotent Jordan block of size \(n\), one has
\(S_n^n=0\) and \(S_n^{n-1}\ne0\). Consequently,
\[
\spec(A_q)=\{0\},
\qquad
m_{A_q}(x)=x^n,
\qquad
d(A_q)=n.
\]
Because the inertia of \(J_q\) is \((q+1,q)\),
\[
\nu_{J_q}(A_q)
=
\min\{n,2q+1,2(q+1)+1\}
=
n.
\]

For \(0\le k\le n-1\), the matrix \(S_n^k\) maps $\operatorname{span}\{e_{k+1},\ldots,e_n\}$ isometrically onto $\operatorname{span}\{e_1,\ldots,e_{n-k}\}$ and vanishes on its orthogonal complement. Hence \(\norm{S_n^k}=1\). Moreover, \(A_q^k=0\) for \(k\ge n\). Since
\(\alpha_n>1\), it follows that
\begin{equation}\label{eq:sharpness-power}
\Power(A_q)=\alpha_n^{n-1}.
\end{equation}

It remains to estimate the Kreiss constant. For \(z\ne0\), nilpotency
gives the finite resolvent expansion
\[
(zI-A_q)^{-1}
=
\sum_{k=0}^{n-1}\frac{\alpha_n^k}{z^{k+1}}S_n^k.
\]
Let \(r:=|z|>1\). Since \(\norm{S_n^k}=1\), the triangle inequality gives
\[
(r-1)\norm{(zI-A_q)^{-1}}
\le
\alpha_n^{n-1}F_n(r),
\]
where
\[
F_n(r):=
\frac{r-1}{r^n}
\sum_{\ell=0}^{n-1}\left(\frac{r}{\alpha_n}\right)^\ell.
\]

Set \(f_n(r):=(r-1)/r^n\). Its derivative is
\[
f_n'(r)=\frac{n-(n-1)r}{r^{n+1}},
\]
so \(f_n\) attains its maximum over \(r>1\) at
\[
r_n:=\frac{n}{n-1}.
\]
The maximal value is
\[
c_n:=f_n(r_n)=\frac{(n-1)^{n-1}}{n^n}.
\]

We now estimate \(F_n(r)\) in four ranges. If \(1<r\le2\), then
\(f_n(r)\le c_n\) and
\[
\sum_{\ell=0}^{n-1}
\left(\frac{r}{\alpha_n}\right)^\ell
\le
\frac{1}{1-2/\alpha_n},
\]
so \(F_n(r)\le c_n/(1-2/\alpha_n)\).

If \(2<r\le\alpha_n/2\), then
\[
\sum_{\ell=0}^{n-1}
\left(\frac{r}{\alpha_n}\right)^\ell
\le2
\]
and \(f_n(r)\le r^{1-n}\le2^{1-n}\). Hence
\(F_n(r)\le2^{2-n}\).

If \(\alpha_n/2<r\le\alpha_n\), then each term in the sum defining
\(F_n(r)\) is at most \(1\). Therefore,
\[
F_n(r)\le n r^{1-n}
\le n\left(\frac{2}{\alpha_n}\right)^{n-1}.
\]

Finally, if \(r\ge\alpha_n\), then
\[
\sum_{\ell=0}^{n-1}
\left(\frac{r}{\alpha_n}\right)^\ell
\le
n\left(\frac{r}{\alpha_n}\right)^{n-1},
\]
and hence \(F_n(r)\le n/\alpha_n^{n-1}\).

Since
\[
c_n
=
\frac1n\left(1-\frac1n\right)^{n-1}
\ge
\frac{1}{en}
\]
and \(\alpha_n=n^3\), the estimates in the last three ranges satisfy
\[
\frac{2^{2-n}}{c_n}\le en2^{2-n}\to0, \qquad \frac{n(2/\alpha_n)^{n-1}}{c_n} \le
en^2\left(\frac{2}{n^3}\right)^{n-1}
\to0,
\]
and
\[
\frac{n/\alpha_n^{n-1}}{c_n}
\le
\frac{en^2}{n^{3(n-1)}}
\to0.
\]
Moreover, \((1-2/\alpha_n)^{-1}=1+o(1)\). Therefore,
\[
\sup_{r>1}F_n(r)\le(1+o(1))c_n,
\]
where \(o(1)\) is understood as \(n\to\infty\). It follows that
\begin{equation}\label{eq:sharpness-kreiss-upper}
\Kreiss(A_q)
\le
(1+o(1))\alpha_n^{n-1}c_n.
\end{equation}

For the reverse estimate, evaluate the resolvent at the positive real
number \(z=r_n\). Write
\[
(r_nI-A_q)^{-1}
=
\alpha_n^{n-1}r_n^{-n}S_n^{n-1}+E_n,
\]
where
\[
E_n:=
\sum_{k=0}^{n-2}\alpha_n^kr_n^{-k-1}S_n^k.
\]
Since \(\norm{S_n^{n-1}}=1\), the reverse triangle inequality gives
\[
\norm{(r_nI-A_q)^{-1}}
\ge
\alpha_n^{n-1}r_n^{-n}-\norm{E_n}.
\]
Furthermore,
\[
\frac{\norm{E_n}}
{\alpha_n^{n-1}r_n^{-n}}
\le
\sum_{\ell=1}^{n-1}
\left(\frac{r_n}{\alpha_n}\right)^\ell
\le
\frac{r_n/\alpha_n}{1-r_n/\alpha_n}
=o(1).
\]
Therefore,
\[
\Kreiss(A_q)
\ge
(r_n-1)\norm{(r_nI-A_q)^{-1}}
\ge
(1-o(1))\alpha_n^{n-1}c_n.
\]
Together with \eqref{eq:sharpness-kreiss-upper}, this yields
\begin{equation}\label{eq:sharpness-kreiss-asymptotic}
\Kreiss(A_q)
=
(1+o(1))\alpha_n^{n-1}c_n.
\end{equation}

Combining \eqref{eq:sharpness-power} and
\eqref{eq:sharpness-kreiss-asymptotic}, we obtain
\[
\frac{\Power(A_q)}{\Kreiss(A_q)}
=
(1+o(1))\frac{1}{c_n}
=
(1+o(1))
n\left(1+\frac{1}{n-1}\right)^{n-1}
\sim en.
\]
Since \(n=2q+1\), this proves
\[
\frac{\Power(A_q)}
{e(2q+1)\Kreiss(A_q)}
\longrightarrow1.
\]

If \(N=n\), take \(A_{q,N}:=A_q\) and \(J_{q,N}:=J_q\); the required properties have already been proved. Assume now that \(N>n\), set \(m:=N-n\), and define
\[
A_{q,N}:=A_q\oplus0_m,
\qquad
J_{q,N}:=J_q\oplus I_m.
\]
Then \(J_{q,N}\) has inertia \((N-q,q)\), and
\(A_{q,N}\) is \(J_{q,N}\)-selfadjoint. Moreover,
\(\spec(A_{q,N})=\{0\}\).

For every \(k\ge0\),
\[
\norm{A_{q,N}^k}
=
\max\{\norm{A_q^k},\norm{0_m^k}\},
\]
where \(0_m^0=I_m\). Since \(\Power(A_q)>1\), it follows that
\[
\Power(A_{q,N})=\Power(A_q).
\]
For \(|z|>1\),
\[
(zI-A_{q,N})^{-1}
=
(zI-A_q)^{-1}\oplus z^{-1}I_m.
\]
Thus
\[
\Kreiss(A_{q,N})
=
\max\{\Kreiss(A_q),\Kreiss(0_m)\}
=
\Kreiss(A_q),
\]
because \(\Kreiss(0_m)=1\) and \(\Kreiss(A_q)\ge1\). The asserted limit
therefore holds for every sequence \(N(q)\ge2q+1\).
\end{proof}

Since \(d(A_q)=2q+1\) and \(J_q\) has inertia \((q+1,q)\), one has
\[
\nu_{J_q}(A_q)
=
\min\{2q+1,2q+1,2q+3\}
=
2q+1.
\]
Thus the core construction is asymptotically sharp for the
effective-dimension estimate.

\begin{corollary}
\label{cor:no-sublinear-inertia}
There do not exist a constant \(C>0\) and a function
\(g:\mathbb N\to(0,\infty)\) satisfying \(g(r)=o(r)\) as \(r\to\infty\)
such that
\[
\Power(A)
\le
C\,g\!\left(
\min\{\operatorname{ind}_-(J),\operatorname{ind}_+(J)\}
\right)\Kreiss(A)
\]
for every fundamental symmetry \(J\) and every \(J\)-selfadjoint matrix
\(A\) with \(\spec(A)\subset\Dbar\).
\end{corollary}

\begin{proof}
Suppose that such \(C\) and \(g\) exist. Apply the asserted estimate to
the core matrices \(J_q,A_q\) from Theorem~\ref{thm:sharpness}. Since the
inertia of \(J_q\) is \((q+1,q)\), one obtains
\[
\frac{\Power(A_q)}{\Kreiss(A_q)}
\le
Cg(q)
=
o(q).
\]
On the other hand, Theorem~\ref{thm:sharpness} gives
\[
\frac{\Power(A_q)}{\Kreiss(A_q)}
\sim
e(2q+1),
\]
which is of linear order in \(q\). This contradiction proves the result.
\end{proof}

The core construction satisfies \(d(A_q)=2q+1\), so it does not by itself
separate the inertia term from the minimal-polynomial term in
\(\nu_{J_q}(A_q)\). The next result shows that the same sharpness ratio
persists even when the minimal-polynomial degree equals the full ambient
dimension.

\begin{corollary}
\label{cor:sharpness-full-degree}
For every \(q\ge1\) and \(N\ge2q+1\), there exist a fundamental symmetry
\(\widehat J_{q,N}\in\C^{N\times N}\) with negative index \(q\) and a
\(\widehat J_{q,N}\)-selfadjoint matrix
\(\widehat A_{q,N}\in\C^{N\times N}\) such that
\[
\spec(\widehat A_{q,N})\subset\D,
\qquad
d(\widehat A_{q,N})=N,
\qquad
\nu_{\widehat J_{q,N}}(\widehat A_{q,N})=2q+1.
\]
Moreover, for every sequence \(N(q)\ge2q+1\),
\[
\frac{\Power(\widehat A_{q,N(q)})}
{e(2q+1)\Kreiss(\widehat A_{q,N(q)})}
\xrightarrow[q\to\infty]{}1.
\]
In particular, \(N(q)\) may be chosen so that \(q/N(q)\to0\).
\end{corollary}

\begin{proof}
Let \(A_q,J_q\in\C^{(2q+1)\times(2q+1)}\) be the core matrices from
Theorem~\ref{thm:sharpness}, and set \(m:=N-(2q+1)\). If \(m=0\), take
\(\widehat A_{q,N}:=A_q\) and \(\widehat J_{q,N}:=J_q\).

Assume \(m\ge1\). Choose distinct numbers
\(0<\lambda_1<\cdots<\lambda_m<1\), and define
\[
D_m:=\operatorname{diag}(\lambda_1,\ldots,\lambda_m),\qquad
\widehat A_{q,N}:=A_q\oplus D_m,\qquad
\widehat J_{q,N}:=J_q\oplus I_m.
\]
Then \(\widehat J_{q,N}\) has inertia \((N-q,q)\), and
\(\widehat A_{q,N}\) is \(\widehat J_{q,N}\)-selfadjoint. Since
\(\spec(A_q)=\{0\}\) and \(0<\lambda_j<1\),
\(\spec(\widehat A_{q,N})\subset\D\).

Because \(D_m\) is a Hermitian contraction,
\(\norm{D_m^k}=\lambda_m^k\le1\) for \(k\ge1\), while
\(\norm{D_m^0}=1\). Hence \(\Power(D_m)=1\). Moreover, for \(|z|>1\),
\[
(|z|-1)\norm{(zI-D_m)^{-1}}
=
\max_{1\le j\le m}
\frac{|z|-1}{|z-\lambda_j|}
\le
\frac{|z|-1}{|z|-\lambda_m}
\le1.
\]
The general lower bound \(\Kreiss(D_m)\ge1\) therefore gives
\(\Kreiss(D_m)=1\). It follows from the direct-sum structure that
\[
\Power(\widehat A_{q,N})=\Power(A_q),
\qquad
\Kreiss(\widehat A_{q,N})=\Kreiss(A_q).
\]

The minimal polynomial of \(A_q\) is \(x^{2q+1}\), while
\[
m_{D_m}(x)=\prod_{j=1}^m(x-\lambda_j).
\]
These two polynomials are coprime because each \(\lambda_j\) is nonzero.
The minimal polynomial of a direct sum is the least common multiple of
the minimal polynomials of its blocks. Therefore,
\[
m_{\widehat A_{q,N}}(x)
=
x^{2q+1}\prod_{j=1}^m(x-\lambda_j),
\]
and hence \(d(\widehat A_{q,N})=2q+1+m=N\).

Finally, since \(N\ge2q+1\),
\[
\nu_{\widehat J_{q,N}}(\widehat A_{q,N})
=
\min\{N,2q+1,2(N-q)+1\}
=
2q+1.
\]
The asymptotic limit now follows directly from
Theorem~\ref{thm:sharpness}.
\end{proof}

\section{Additional results}\label{sec:consequences}

The main estimate is stable under two natural transformations: radial scaling of the disk and reduction of a nonsingular Hermitian metric to a fundamental symmetry.

\begin{corollary}\label{cor:geometric-decay}
Under the assumptions of Theorem~\ref{thm:main}, suppose additionally that $\rho(A)<\gamma<1$. Then, for every $k\ge1$,
\[
\norm{A^k}\le\nu_J(A)\left(1+\frac1k\right)^k\Kreiss_\gamma(A)\gamma^k\le e\nu_J(A)\Kreiss_\gamma(A)\gamma^k.
\]
\end{corollary}

\begin{proof}
The matrix $A/\gamma$ is $J$-selfadjoint and has spectrum in $\D$. With the change of variables $z=\gamma w$,
\[
\begin{aligned}
\Kreiss(A/\gamma)
&=\sup_{|w|>1}(|w|-1)\norm{(wI-A/\gamma)^{-1}}\\
&=\sup_{|z|>\gamma}\frac{|z|-\gamma}{\gamma}\norm{\gamma(zI-A)^{-1}}\\
&=\Kreiss_\gamma(A).
\end{aligned}
\]
Apply Theorem~\ref{thm:main} to $A/\gamma$. Since $d(A/\gamma)=d(A)$, one has $\nu_J(A/\gamma)=\nu_J(A)$. Using $A^k=\gamma^k(A/\gamma)^k$ gives the first estimate. The second follows from $(1+1/k)^k\le e$.
\end{proof}

\begin{definition}\label{def:weighted-kreiss}
Let $G=G^*$ be nonsingular with inertia $(n-q,q)$. By the spectral theorem, write $G=U\operatorname{diag}(\lambda_1,\ldots,\lambda_n)U^*$, where $U$ is unitary and each $\lambda_j$ is real and nonzero. For a nonzero real number $t$, let $\operatorname{sign}(t)=1$ when $t>0$ and $\operatorname{sign}(t)=-1$ when $t<0$. Define
\[
|G|:=U\operatorname{diag}(|\lambda_1|,\ldots,|\lambda_n|)U^*,\qquad
J:=\operatorname{sign}(G):=U\operatorname{diag}(\operatorname{sign}\lambda_1,\ldots,\operatorname{sign}\lambda_n)U^*,
\]
and set $R:=|G|^{1/2}$. Then $|G|$ and $R$ are positive definite, while $J$ is a fundamental symmetry with the same inertia as $G$. Moreover, $R$ and $J$ commute and $G=RJR$. The norm induced by $|G|$ on $\C^n$ and the corresponding operator norm are defined by
\[
\lVert x\rVert_{|G|}:=\norm{Rx},
\qquad
\lVert T\rVert_{|G|}
:=
\sup_{x\ne0}
\frac{\lVert Tx\rVert_{|G|}}{\lVert x\rVert_{|G|}}
=
\norm{RTR^{-1}}.
\]
For $A\in\C^{n\times n}$, define
\[
\Power_{|G|}(A) := \sup_{k\ge0}\lVert A^k\rVert_{|G|} \qquad \text{and} \qquad \Kreiss_{|G|}(A)
:=
\sup_{|z|>1}
(|z|-1)\lVert(zI-A)^{-1}\rVert_{|G|}.
\]
If $\rho(A)<\gamma$, define the scaled weighted Kreiss constant by
\[
\Kreiss_{\gamma,|G|}(A)
:=
\sup_{|z|>\gamma}
(|z|-\gamma)\lVert(zI-A)^{-1}\rVert_{|G|}.
\]
\end{definition}

\begin{corollary}\label{cor:metric}
Let $G=G^*$ be nonsingular with inertia $(n-q,q)$, and suppose that $A^*G=GA$. Set $R:=|G|^{1/2}$ and $J:=\operatorname{sign}(G)$. 
If $\spec(A)\subset\Dbar$, then
\[
\Power_{|G|}(A)\le e\min\{d(A),2q+1,2(n-q)+1\}\Kreiss_{|G|}(A).
\]
If $\rho(A)<\gamma<1$, then, for every $k\ge1$,
\[
\lVert A^k\rVert_{|G|}\le\min\{d(A),2q+1,2(n-q)+1\}\left(1+\frac1k\right)^k\Kreiss_{\gamma,|G|}(A)\gamma^k.
\]
In particular,
\[
\lVert A^k\rVert_{|G|}\le e\min\{d(A),2q+1,2(n-q)+1\}\Kreiss_{\gamma,|G|}(A)\gamma^k.
\]
\end{corollary}

\begin{proof}
By Definition~\ref{def:weighted-kreiss}, $R$ is positive definite, $J$ is a fundamental symmetry with inertia $(n-q,q)$, and $G=RJR$. Set $B:=RAR^{-1}$. The identity $A^*G=GA$ is
\[
A^*RJR=RJRA.
\]
Multiplying this identity on the left and on the right by $R^{-1}$ gives
\[
R^{-1}A^*RJ=JRAR^{-1}.
\]
Since $B^*=R^{-1}A^*R$, it follows that $B^*J=JB$. Thus $B$ is $J$-selfadjoint.

Since $B$ is similar to $A$, one has $\spec(B)=\spec(A)$, $\rho(B)=\rho(A)$, and $d(B)=d(A)$. Since $B^k=RA^kR^{-1}$ for every $k\ge0$, one also has $\Power(B)=\Power_{|G|}(A)$.
For every $z\notin\spec(A)$, the identity $(zI-B)^{-1}=R(zI-A)^{-1}R^{-1}$ gives $\Kreiss(B)=\Kreiss_{|G|}(A)$ and $\Kreiss_\gamma(B)=\Kreiss_{\gamma,|G|}(A)$.
The unscaled estimate follows from Theorem~\ref{thm:main} applied to $B$. The scaled pointwise estimate follows from Corollary~\ref{cor:geometric-decay} applied to $B$. The final estimate follows from $(1+1/k)^k\le e$.
\end{proof}

\section*{Declaration of competing interest}
The authors declare no competing interests.

\section*{Data availability}
No data were used for the research described in this article.

\section*{Funding}
This work is supported by VinUni's grant VUNI.2526.CAIR.04 and Seed Fund 10000225 CECS.

\bibliographystyle{plainnat}
\bibliography{references}

\end{document}